\documentclass[12pt]{amsart}
 
\usepackage {amsfonts, amssymb, amscd, amsthm, amsmath, xypic}
\usepackage[dvips]{graphicx}
\usepackage{geometry}
\newtheorem{Thm}{Theorem}[section]
\newtheorem{theorem}[Thm]{Theorem}

\newtheorem{proposition}[Thm]{Proposition}

\newtheorem{remark}[Thm]{Remark}

\renewcommand{\phi}{\varphi}
\newcommand{\co}{\mathbb{C}}

\newcommand{\D}{\mathbb{D}}
\newcommand{\Dh}{\widehat{\mathbb{D}}}
\newcommand{\T}{\mathbb{T}}
\newcommand{\tz}{\mathbb{T}}

\newcommand{\rea}{{\rm Re}\,}

\title{Hypercyclic Toeplitz operators}
\author{Anton Baranov, Andrei Lishanskii}

\address{ 
 Anton Baranov, 
\newline
Department of Mathematics and Mechanics,
St. Petersburg State University, 
\newline
St. Petersburg, Russia,
\newline
\phantom{x}\,\, and
\newline
National Research University  Higher School of Economics,
\newline
St. Petersburg, Russia,
\newline {\tt anton.d.baranov@gmail.com}
\newline\newline \phantom{x}\,\, Andrei Lishanskii,
\newline 
Department of Mathematics and Mechanics,
St. Petersburg State University, 
\newline 
St. Petersburg, Russia,
\newline
\phantom{x}\,\, and
\newline 
Chebyshev Laboratory,
St. Petersburg State University,
\newline 
St. Petersburg, Russia,
\newline {\tt lishanskiyaa@gmail.com}
\newline\newline \phantom{x}
}

\thanks{The authors were supported by the grant MD-5758.2015.1. 
The second author was partially supported by JSC "Gazprom Neft" and 
by RFBR grant 14-01-31163.}



\begin{document}

\begin{abstract}
We study hypercyclicity of the Toeplitz operators in the Hardy space 
$H^2(\mathbb{D})$ 
with symbols of the form $p(\overline{z}) +\phi(z)$, where $p$ is a polynomial
and $\phi \in H^\infty(\mathbb{D})$. 
We find both necessary and sufficient conditions for hypercyclicity
which almost coincide in the case when ${\rm deg}\, p =1$.
\end{abstract}

\maketitle

\sloppy
%\baselinestretched

%\large
%\normalsize

\section{Introduction and Main Results}

Toeplitz operators with antianalytic symbols are among the
basic examples of hypercyclic operators. In 1968, 
S. Rolewicz showed that the operator $T_{\alpha\overline{z}}$ 
(a multiple of the backward shift) is hypercyclic on the Hardy space $H^2$ 
whenever $|\alpha|>1$. Later, G. Godefroy and J. Shapiro \cite{gosh} 
showed that for a function 
$\phi \in H^\infty$ the antianalytic Toeplitz operator $T_{\overline{\phi}}$
is hypercyclic if and only if $\phi(\mathbb{D}) \cap \mathbb{T} \ne \emptyset$.
Here, as usual, $\mathbb{D}$ and $\mathbb{T}$ denote the unit disc and 
the unit circle, respectively. On the other hand, it is obvious that there are no 
hypercyclic Toeplitz operators with analytic symbols
(i.e., among multiplication operators). 

However, it seems that hypercyclicity phenomena for general
Toeplitz operators are much less studied, and the hypercyclicity criteria 
are not known. This problem was explicitly stated by Shkarin \cite{sh} who described
hypercyclic Toeplitz operators with symbols of the form $\Phi(z) =a\overline{z} +b +cz$
(i.e., with tridiagonal matrix).

The aim of this note is to give new examples of hypercyclic Toeplitz operators.
We give necessary or sufficient conditions 
for hypercyclicity of $T_\Phi$ in the case when 
\begin{equation}
\label{fg}
\Phi(z) = p\bigg(\frac{1}{z}\bigg) +\phi(z),
\end{equation}
where $p$ is a polynomial and $\phi$ is in $H^\infty$ 
(sometimes we will assume that $\phi$ belongs to 
the disc-algebra $A(\mathbb{D})$). 
In the case when $p(z) = \gamma z$ (i.e., $\Phi \in \overline{z} H^\infty$) 
the gap between the necessary and sufficient 
conditions becomes especially small. 

A novel feature of these conditions is the role of univalence or $N$-valence
(where $N$ is the degree of $p$) of the symbol. 
It seems that such conditions did not appear
in the linear dynamics before, with one notable exception: in \cite{bs} 
Bourdon and Shapiro studied Bergman space Toeplitz operators with {\it antianalytic} symbols
and in some of their results the univalence of the symbol plays a role.
\bigskip

Let us state the main results of the paper. 
In what follows we denote by $\overline{\D}$ the closed unit disc and
put $\widehat{\D} = \co\setminus \overline{\D}$.

Our first result applies to the case when the antianalytic part has degree 1.

\begin{theorem}
\label{main}
Let $\gamma\in \co$, let $\phi \in H^\infty$
and let $\Phi(z) = \frac{\gamma}{z} + \phi(z)$.
\medskip
\\
1. If $T_{\Phi}$ is hypercyclic, then 

$(a)$ the function $\Phi$ is univalent in $\D \setminus \{0\}$;

$(b)$ $\overline{\D} \cap (\co \setminus \Phi(\D)) \ne \emptyset$ and 
$\Dh \cap (\co \setminus \Phi(\D)) \ne \emptyset$.
\medskip
\\
2. Assume that $\phi \in A(\overline{\D})$ and that

$(a')$ the function $\Phi$ is univalent in $\overline{\D}\setminus\{0\}$;

$(b')$ $\D \cap (\co \setminus \Phi(\D)) \ne \emptyset$ and 
$\Dh \cap (\co \setminus \Phi(\D)) \ne \emptyset$.

Then $T_{\Phi}$ is hypercyclic.
\end{theorem}

The gap between the necessary and sufficient conditions is related only
to the boundary behaviour of $\Phi$. While it is necessary that $\Phi$ 
is univalent in $\D$, we ask for univalence up to the boundary in the 
sufficient condition. Also, while the necessary condition requires the spectrum
$\sigma(T_\Phi) = \co\setminus \Phi(\D)$ 
to intersect the unit circle, in the sufficient condition we need a stronger
assumption that the set $\co\setminus \overline{\Phi(\D)}$ (which is, essentially,
the point spectrum of $T_\Phi$) intersects the open disc $\D$.
\medskip

In our second result $p$ is a polynomial of degree $N$. Recall that an 
analytic function $h$ in the domain
$D$ is said to be $N$-valent in $D$ if the equation $h(z) = w$ has at most $N$ solutions
in $D$ counting multiplicities. 
Note that $\Phi(z) \sim c_N z^{-N}$, $z\to 0$, and so $\Phi(z) = w$ has exactly 
$N$ solutions when $|w|$ is sufficiently large. Put
$$
\Phi(\D, N) = \{w\in \co: \text{equation } \Phi(z) = w
\text{ has exactly } N \text{ solutions in }\D \},
$$
where the solutions are counted according to their multiplicities.

\begin{theorem}
\label{main1}
Let $p$ be a polynomial of degree $N\ge 1$, let $\phi \in H^\infty$
and let $\Phi$ be given by \eqref{fg}.
\medskip
\\
1. If $T_{\Phi}$ is hypercyclic, then 

$(a)$ the function $\Phi$ is $N$-valent in $\D\setminus\{0\}$;

$(b)$ $\overline{\D} \cap (\co \setminus \Phi(\D, N)) \ne \emptyset$ and 
$\Dh \cap (\co \setminus \Phi(\D, N)) \ne \emptyset$.
\medskip
\\
2. Assume that $\phi \in A(\overline{\D})$ and that  

$(a')$ for any $w \in 
\Phi(\overline{\D}\setminus\{0\})$ the equation $\Phi(z) = w$
has exactly $N$ solutions in $\overline{\D}\setminus\{0\}$;

$(b')$ $\D \cap (\co \setminus \Phi(\D))$ and
$\Dh \cap (\co\setminus\Phi(\D)) \ne \emptyset$. 

Then $T_{\Phi}$ is hypercyclic.
\end{theorem}

Note that condition $(a')$ implies, in particular, that
$\Phi(\mathbb{D}) = \Phi(\mathbb{D}, N)$.
\medskip

The proofs of Theorems \ref{main} and \ref{main1} are essentially
elementary (modulo some basic results about polynomial 
approximation, like Mergelyan's theorem).
\bigskip

\section{Preliminaries}

Recall that a continuous linear operator $T$ in a separable Banach (or Fr\'echet) space 
$X$  is said to be  \textit{hypercyclic} if 
there exists $x \in X$ such that the set 
$\{T^n x, n\in\mathbb{N}_0\}$ is dense in $X$
(here $\mathbb{N}_0 = \{0,1,2, \dots\}$). 

One of the most basic sufficient conditions of hypercyclicity is the so-called 
Godefroy--Shapiro criterion (see \cite{gosh} or \cite{bm, gp}). 
Suppose that, for a continuous linear operator $T$, the subspaces 
$$
\begin{aligned}
X_0 & = {\rm span}\{ x \in X: Tx = \lambda x \text{ for some }
\lambda \in \mathbb{C}, \ |\lambda| < 1 \}, \\
Y_0 & = {\rm span}\{ x \in X: Tx = \lambda 
x \text{ for some } \lambda \in \mathbb{C}, \ |\lambda| > 1 \},
\end{aligned}
$$
are dense in $X$. Then $T$ is hypercyclic.

Let $H^2$ denote the standard Hardy space in $\mathbb{D}$.
Recall that for a function $\psi \in L^\infty(\mathbb{T})$ the Toeplitz operator
$T_\psi$ with the symbol $\psi$ is defined as 
$T_\psi f = P_+(\psi f)$, where $P_+$ stands for the orthogonal projection from
$L^2(\mathbb{T})$ onto $H^2$. 

In this section we always assume that 
$T_\Phi$ is the Toeplitz operator with the symbol \eqref{fg}, where $p$ 
is a polynomial of degree $N\ge 1$ and $\phi \in H^\infty$. 
Without loss of generality we assume that $p(0) = 0$. 

First we show that $N$-valence of $\Phi$ in $\D$ is necessary for hypercyclicity.

\begin{proposition}
\label{kai}
Assume that for some $\mu\in \co$, the equation 
$\Phi(z) = \mu$ has at least $N+1$ solutions counting multiplicities.
Then $(T_\Phi)^* = T_{\overline{\Phi}}$ 
has an eigenvector and, in particular, $T_\Phi$ is not hypercyclic.
\end{proposition}

\begin{proof}
Denote by $k_\lambda$ the Cauchy kernel (reproducing kernel of $H^2$):
$k_\lambda(z) = \frac{1}{1-\bar \lambda z}$. It is well known that for any 
antianalytic Toeplitz operator we have $T_{\overline{\phi}} k_\lambda = 
\overline{\phi(\lambda)} k_\lambda$. 

Assume, for simplicity, that the equation  $\Phi(z) = \mu$ has $N+1$
distinct solutions $z_1, z_2, \dots z_{N+1}$ in $\D$. 
We will construct
the eigenvector of $(T_\Phi)^*$ in the form $f = \sum_{j=1}^{N+1}\alpha_j k_{z_j}$,
where $\alpha_j$ are some complex coefficients. If $p(z) = \sum_{k=1}^N c_k z^k$, then
$T_{\overline{\Phi}} = T_{\bar p(z) +\overline{\phi(z)}}$, where
$\bar p(z) = \sum_{k=1}^N \bar c_k z^k$. Hence,
using the fact that $p(1/z_j) + \phi(z_j) = \mu$, we get
\begin{equation}
\label{bab2}
T_{\overline{\Phi}} f(z) = \sum_{j=1}^{N+1} \alpha_j\bigg(
\frac{p(z)}{1-\overline{z_j} z} +\frac{\overline{\phi(z_j)}}{1-\overline{z_j} z}\bigg) = 
\bar \mu f(z) +\sum_{j=1}^{N+1}\alpha_j 
\frac{\bar p(z) - \overline{p(1/z_j)}}{1-\overline{z_j} z}.
\end{equation}
The functions 
$$
\frac{\bar p(z) - \overline{p(1/z_j)}}{1-\overline{z_j} z} =
\frac{\bar p(z) - \bar p (1/\overline{z_j})}{1-\overline{z_j} z},
\qquad j=1,\dots, N+1, 
$$
are polynomials of degree $N-1$. Hence, there exist nontrivial
coefficients $\alpha_j$ such that the last sum in \eqref{bab2} is 
identically zero, and so 
$T_{\overline{\Phi}} f = \bar \mu f$.

In the case of a zero $z_j$ of multiplicity $m_j$, consider the linear 
combination of the functions $(1-\bar z_j z)^{-l}$, $1\le l\le m_j$. We 
omit the straightforward computations.
\end{proof}

Next we study the  spectrum $\sigma(T_\Phi)$, 
the point spectrum $\sigma_p(T_\Phi)$ and its eigenvectors.

Note that $T_{\bar z} = T_{1/z}$ is the backward shift operator $S^*$ 
on $H^2$, that is,
$$
T_{\bar z}f = \frac{f(z)-f(0)}{z} \qquad \text{and} \qquad
T_{\bar z^k}f = \frac{1}{z^k}\bigg( f(z) - \sum_{j=0}^{k-1} \frac{f^{(j)}(0)}{j!}z^j\bigg).
$$

In the proof of the next proposition we will need the basic results
on inner-outer (Nevanlinna) factorization of the functions
in the Hardy spaces (see, e.g., \cite[Chapter 2]{dur} or \cite[Chapter IV]{koo}).
  
\begin{proposition}
\label{sob}
Assume that $\Phi$ is $N$-valent in $\D$. Then 
$$
\sigma(T_\Phi) = \co\setminus \Phi(\D, N), \qquad 
\sigma_p(T_\Phi)\supset \co\setminus \overline{\Phi(\D)}.
$$
If $\lambda \in \co\setminus \overline{\Phi(\D)}$, then 
the corresponding eigenspace has dimension $N$ 
and the eigenvectors are given by
$$
f_\lambda(z)  = \frac{q(z)}{z^N \Phi (z) - \lambda z^N}, 
$$
where $q$ is an arbitrary polynomial of degree at most $N-1$.
\end{proposition}

\begin{proof} First we prove 
the inclusion $\sigma(T_\Phi) \subset \co\setminus \Phi(\D, N)$.
Namely, we show that any $\lambda \in \Phi(\D, N)$ 
is a regular point for $T_\Phi$, i.e., the 
equation $T_\Phi f - \lambda f =g$ has the unique solution  $f\in H^2$
for any $g\in H^2$.

Let $p(z) = \sum_{k=1}^N c_k z^k$. Then the equation $T_\Phi f - \lambda f =g$
may be rewritten as
$$
\sum_{k=1}^N  \frac{c_k}{z^k} 
\bigg( f(z) - \sum_{j=0}^{k-1} \frac{f^{(j)}(0)}{j!}z^j\bigg) +\phi(z) 
f(z) - \lambda f(z) = g(z),
$$ 
or, equivalently, 
$$
f(z) \bigg(\sum_{k=1}^N c_k z^{N-k} +z^N \phi(z)- \lambda z^N\bigg) = 
z^N g(z) + \sum_{k=1}^N c_k \sum_{j=0}^{k-1}\frac{f^{(j)}(0)}{j!}z^{N-k+j}.
$$
If $\lambda \in \Phi(\D, N)$, then the expression in brackets 
(which equals $z^N \Phi (z) - \lambda z^N$) has exactly $N$ zeros
in $\D$ counting multiplicities, say, $z_1, z_2, \dots, z_N$.
Moreover, it is clear that $|\Phi (z) - \lambda| \ge \delta>0$
for some $\delta>0$ and almost every $z \in\mathbb{T}$.   
Consider the (unique) polynomial $q$ of degree $N-1$ 
such that $z_j^N g(z_j)+q(z_j)=0 $, $j=1, \dots N$
(with obvious modification for multiple zeros).
Then, for this choice of $q$, the function 
\begin{equation}
\label{lish}
f(z)  = \frac{z^N g(z) + q(z)}{z^N \Phi (z) - \lambda z^N}
\end{equation}
belongs to $H^2$. Note that 
we necessarily have 
\begin{equation}
\label{lish1}
q(z) = \sum_{k=1}^N c_k \sum_{j=0}^{k-1}\frac{f^{(j)}(0)}{j!}z^{N-k+j}
\end{equation}
(just compare the Taylor coefficients), and so $f$ is indeed the unique solution 
of the equation $T_\Phi f - \lambda f = g$.
Thus, we have shown that $\sigma(T_\Phi) \subset \co\setminus \Phi(\D, N)$.
\bigskip

For the proof of the converse inclusion $\co\setminus \Phi(\D, N) \subset \sigma(T_\Phi)$
we will need the following observation:
\medskip
\\
{\bf Claim.} {\it If $\lambda$ is a regular point for $T_\Phi$, then 
$(\Phi - \lambda)^{-1} \in L^\infty(\mathbb{T})$ and 
the Nevanlinna factorization of the 
function $\Psi(z) = z^N\Phi(z) - \lambda z^N \in H^\infty$ contains no nontrivial 
singular inner factor.}
\medskip
\\
{\it Proof of the Claim.} Assume 
that the equation $T_\Phi f - \lambda f = g$ has the unique solution
for any $g\in H^2$. Then $f$ is of the form \eqref{lish} where the polynomial $q$
is given by \eqref{lish1}. Note that for a function $\gamma$ on $\mathbb{T}$ 
we have the inclusion $\gamma H^2 = \{\gamma h: h\in H^2\} 
\subset L^2(\mathbb{T})$ if and only if 
$\gamma\in L^\infty(\mathbb{T})$. Since the function $f$ in \eqref{lish} 
is in $H^2$ for any $g\in H^2$, while 
$\|q\|_\infty \le C\|f\|_2 \le C_1\|g\|_2$ 
for some constants $C, C_1$ independent from $g$, 
we conclude that $(z^N \Phi (z) - \lambda z^N)^{-1} \in 
L^\infty(\mathbb{T})$.

If $\Psi$ has a nontrivial singular inner factor, then, taking,   
$g\equiv 1$, we obtain a function $f$ of the form $f = \frac{u_1B_1}{u_2B_2I}$,
where $u_1, u_2$ are outer functions, $B_1, B_2$ Blaschke products
and $I$ a nontrivial singular inner function. Hence, $f\notin H^2$, a contradiction.
The Claim is proved.
\bigskip

Now we return to the proof of the
inclusion $\co\setminus \Phi(\D, N) \subset \sigma(T_\Phi)$.
Let $\lambda \notin \Phi(\D, N)$. We will show that 
$\lambda \in \sigma(T_\Phi)$. 
From now on we assume that $(\Phi - \lambda)^{-1} \in L^\infty(\mathbb{T})$
and the Nevanlinna factorization of $\Psi$ contains no nontrivial 
singular inner factor (otherwise, we already know from the Claim that 
$\lambda \in \sigma(T_\Phi)$).

Assume first that $\lambda \notin \Phi(\D)$. Then 
$\Psi \ne 0$ in $\D$, $\Psi$ has no singular inner factor and so $\Psi$
is an outer $H^\infty$ function. Since $\Psi^{-1} \in L^\infty(\T)$,
we conclude that $\Psi^{-1} \in H^\infty(\D)$. 
Hence, the function
\begin{equation}
\label{bab1}
f_\lambda(z)  = \frac{Q(z)}{\Psi(z)} = \frac{Q(z)}{z^N \Phi (z) - \lambda z^N}
\end{equation}
is in $H^2$ and is an eigenvector of $T_\Phi$ for any choice of the polynomial
$Q$ of degree at most $N-1$. 

Finally, if $\lambda \in \Phi(\D) \setminus\Phi(\D, N)$,
then $\Psi$ has $m$ zeros  
$z_1, z_2, \dots, z_m$ in $\D$ counting multiplicities, where 
$m<N$ (recall that $\Phi$ is $N$-valent in $\mathbb{D}$). 
Therefore, for any polynomial $Q$ which vanishes at $z_j$, 
the function \eqref{bab1} will be an eigenfunction of $T_\Phi$.
Thus, $\co\setminus \Phi(\D, N) \subset \sigma(T_\Phi)$. 
\medskip

The inclusion $\co\setminus \overline{\Phi(\D)} \subset \sigma_p(T_\Phi)$ is easy.
If $\lambda \in \co\setminus \overline{\Phi(\D)}$, then, for some $\delta>0$, we have
$|\Psi(z)|\ge \delta$, $z\in \D$, whence $\Psi^{-1} \in H^\infty(\D)$,
and so any function $f$ of the form \eqref{bab1} is an eigenvector. 
\end{proof}

\begin{remark}
{\rm Note that we have shown in the proof of Proposition \ref{sob}
that $\sigma_p(T_\Phi)$ contains all points $\lambda\in \Phi(\D) \setminus\Phi(\D, N)$
such that $(\Phi - \lambda)^{-1} \in L^\infty(\mathbb{T})$ 
and $\Psi$ has no singular inner factor (this is the case, e.g., if 
there exist $r\in (0,1)$, $\delta>0$ such that $|\Psi(z)| \ge \delta$, $r<|z|<1$). }
\end{remark}
\bigskip

\section{Proofs of main results}

We start with the proof of necessity parts 
of Theorems \ref{main} and \ref{main1}.
\medskip
\\
{\it Proof of Statement 1 of Theorems \ref{main} and \ref{main1}.}  
By Proposition \ref{kai}, if $T_{\Phi}$  
is hypercyclic, then $\Phi$ is $N$-valent in $\D$. In particular, $\Phi$
is univalent in $\D$ when $N=1$. Property (a) is proved. 

Clearly, if $ \Dh \subset \Phi(\D, N)$, then for any $\zeta\in \tz$ for which 
the nontangential boundary value $\Phi(\zeta)$ exists, we have $|\Phi(\zeta)| \le 1$.
Indeed, otherwise there exist $z_1, \dots z_N \in \mathbb{D}$ such that 
$\Phi(z_j) = \Phi(\zeta)$ and the equation $\Phi(z) = w$ will have at least
$N+1$ solutions for some $w$ sufficiently close to $\Phi(\zeta)$.
Hence, $|\Phi|\le 1$ a.e. on $\tz$ and so $\|T_\Phi\|\le 1$, 
a  contradiction to hypercyclicity. 

Finally, if $T_\Phi$ is hypercyclic, then $\sigma(T_\Phi)\cap \tz \ne \emptyset$.
By Proposition \ref{sob}, $\sigma(T_\Phi) = \co\setminus \Phi(\D, N)$
and, in particular, $\sigma(T_\Phi) = \co\setminus\Phi(\D)$
when $N=1$. This completes the proof of (b).  
\qed
\bigskip

The following proposition plays a key role in 
the proof of sufficient conditions in Theorems \ref{main} and \ref{main1}. 

\begin{proposition}
\label{pr1}
1. Let $h\in A(\D)$ be injective in $\overline{\D}$ 
\textup(i.e. univalent up to the boundary\textup). Then the system
$\{h^k\}_{k\ge 0}$ is complete in $H^2$.
\medskip

2. Let $h\in A(\D)$ be $N$-valent in $\overline{\D}$ and, moreover,  
assume that for any $w\in h(\overline{\D})$ the equation $h(z) = w$ has exactly 
$N$ solutions in $\overline{\D}$. Then the system of functions 
$\{z^j h^k: k\ge 0, j=0, 1, \dots, N-1 \}$ is complete in $H^2$.
\end{proposition}

\begin{proof}
1. Let $\Omega = h(\D)$, $\Gamma = \partial \Omega$, $g=h^{-1}: \Omega \to \D$.
Clearly, $g$ admits a continuation to a continuous function on $\overline \Omega = 
\Omega\cup \Gamma$. Since $\Gamma$ is a closed Jordan curve (without intersections), 
the complement $\co\setminus \overline{\Omega}$ is connected and so, by 
Mergelyan's theorem, any function $f$ in $H^\infty(\Omega) \cap 
C(\overline \Omega)$ may be uniformly approximated by analytic polynomials, 
$p_n(u) \to f(u)$ uniformly in $u\in \overline \Omega$. Hence,
$p_n(h(z)) \to f(h(z))$ uniformly in $z\in \overline{\D}$, whence any function 
from $H^\infty\cap C(\overline{\D})$ may be approximated by polynomials in $h$.
\medskip

2. It is not difficult to show that the hypothesis implies that
for any function $f$ which is sufficiently smooth up to the boundary
(say, $f\in C^{N}(\overline{\D})$) there exist functions 
$f_j\in H^\infty (\Omega) \cap C(\overline{\Omega})$, $j=0, 1, \dots, N-1$,
such that 
\begin{equation}
\label{bab5}
f(z) = f_0(h(z)) +z f_1(h(z)) +\dots+ z^{N-1} f_{N-1}(h(z)).
\end{equation}
Here, as above, $\Omega = h(\D)$.
Indeed, for a point $w$ with $N$ distinct preimages $z_1, \dots, z_N$
consider the system of linear equations $f(z_l) = \sum_{j=0}^{N-1}
z_l^j f_j(w)$, $l=1,2,\dots, N$, 
with the unknown $f_j(w)$. Since $z_l$ are locally analytic functions of $w$,
we conclude that $f_j$ are locally analytic at such points $w$; 
it is easy to show that the functions $f_j$ have removable singularities at $w$
in the case of multiple zeros, and so are analytic in the whole $\Omega$ and
continuous up to the boundary.

Now it remains to note that the "exact $N$-valence up to the boundary"$ $ condition 
implies that $\Omega$ is a Jordan domain, 
$\mathbb{C} \setminus \overline{\Omega}$ is connected
and, by Mergelyan's theorem, each function $f_j$ is a 
uniform limit of polynomials $p_{j,m}$, $m\to\infty$, in $\overline{\Omega}$. Hence, 
the sum $\sum_{j=0}^{N-1} z^j p_{j, m}(h(z))$ converges to $f$ uniformly
in $\overline{\D}$. Thus, any sufficiently 
smooth $f$ belongs to the uniform closure in $\overline{D}$ of the linear span
of $\{z^j h^k: k\ge 0, j=0, 1, \dots, N-1 \}$. Hence, this system 
is complete also in $H^2$.
\end{proof}

\begin{remark}
{\rm The problem of completeness of systems $\{h^k\}_{k\ge 0}$ in $H^2(\D)$
or (essentially) equivalent problem of density of polynomials in 
the Hardy space $H^2(\Omega)$, $\Omega = h(\D)$, is in general, a deep problem
for which no explicit answer exists (see \cite{sar, cau, b}).  
Clearly, univalence of $h$ in $\D$ is necessary. 
On the other hand, Caughran \cite{cau} showed that if the polynomials are 
dense in $H^2(\Omega)$ and $h\in A(\D)$, then $\Omega$ is a Jordan domain, and so $h$
is univalent in $\D$ up to the boundary. In the general case it is a result by
Bourdon \cite{b} that the density of polynomials in $H^2(\Omega)$   
implies that $h$  is univalent almost everywhere on~$\tz$.  }
\end{remark}                                                       
\medskip
\noindent
{\it Proof of Statement 2 of Theorems \ref{main} and \ref{main1}.} 
We first consider the case $N=1$, $p(z) =\gamma z$. 
Since, by Proposition~\ref{sob}, $\sigma_p(T_\Phi) \supset  \co\setminus \overline{\Phi(\D)}$,
condition $(b')$ implies that we have 
open sets $U_1\subset \D$ and $U_2\subset \Dh$ of eigenvalues. 
By the Godefroy--Shapiro criterion, 
it remains to show that the corresponding eigenvectors are complete in $H^2$.
Fix some $\lambda_0 \in U_1$ and let 
$$
h(z) = \frac{z}{\gamma - \lambda_0 z + z\phi(z)} = \frac{1}{\Phi(z)-\lambda_0}.
$$
By the conditions on $\Phi$ we have that
$h\in A(\D)$ and $h$ is injective in 
$\overline{\D}$. Now note that for $\lambda$ in a small neighborhood 
$\{|\lambda-\lambda_0|<\delta\}$ of $\lambda_0$
$$
f_\lambda(z) = \frac{1}{\gamma - \lambda z + z\phi(z)} = 
\sum_{k=0}^\infty \frac{(\lambda-\lambda_0)^k z^k}{(\gamma - \lambda_0 z 
+ z\phi(z))^{k+1}}.
$$
Thus, if $f\perp f_\lambda$, $|\lambda-\lambda_0|<\delta$, then
$$
f\perp (\gamma - \lambda_0 z + z\phi(z))^{-1} h^k, \qquad k\ge 0.
$$ 
By Statement 1 of Proposition \ref{pr1}, 
the system $\{h^k\}_{k\ge 0}$ is complete in $H^2$.
The additional factor $1/(\gamma - \lambda_0 z + z\phi(z))$
is an invertible element of $H^\infty$ and so the system 
$$
\{(\gamma - \lambda_0 z + z\phi(z))^{-1} h^k\}_{k\ge 0}
$$
is also complete. We conclude that the eigenvectors corresponding to 
$\lambda\in U_1$ are complete, the proof for $\lambda \in U_2$ is the same.

Now let $N>1$. As above, 
$(b')$ guarantees that we have 
open sets $U_1\subset \D$ and $U_2\subset \Dh$ 
such that $U_1,\, U_2 \subset \co\setminus \overline{\Phi(\mathbb{D})} = 
\co\setminus\overline{\Phi(\mathbb{D}, N)}$ and, thus, consist of eigenvalues.
Fix $\lambda_0 \in U_1$. In this case  we have, by Proposition \ref{sob},
$N$ eigenvectors corresponding to $\lambda_0$, 
$$
f_{\lambda_0, j} (z) = \frac{z^j}{z^N\Phi(z) -\lambda_0 z^N}, \qquad j=0,1, 
\dots, N-1.
$$
Using, as above, the Taylor expansion for $\lambda$ close to $\lambda_0$, we conclude 
that if $f$ is orthogonal to the eigenvectors corresponding to $\lambda$
in a small neighborhood of $\lambda_0$, then
\begin{equation}
\label{bab3}
f\perp \frac{z^j h^k(z)}{z^N\Phi(z) -\lambda_0 z^N}, \qquad k\ge 0,\ 0\le j\le N-1,
\end{equation}
where 
$$
h(z) = \frac{z^N}{z^N \Phi(z) - \lambda_0 z^N} = \frac{1}{\Phi(z) - \lambda_0}.
$$
By the assumptions on $\Phi$, $h$ is $N$-valent and
for any $w\in h(\overline{\D})$ the equation $h(z) = w$ has exactly $N$ solutions in
$\overline{\D}$ counting multiplicities. Hence, by Statement 2 of 
Proposition \ref{pr1},  
the system $\{z^j h^k: k\ge 0, j=0, 1, \dots, N-1 \}$ is complete in $H^2$.
We conclude that any function $f$ satisfying \eqref{bab3} is zero.
\qed

\bigskip

\section{Shkarin's characterization of tridiagonal Toeplitz operators} 

In \cite{sh} Shkarin characterized hypercyclic Toeplitz operators 
with symbols of the form $\Phi(z) = \frac{a}{z} +b +cz$:

\begin{proposition}
\cite[Proposition 5.10]{sh}
The Toeplitz operator $T_\Phi$ with 
$\Phi(z) = \frac{a}{z} +b +cz$ is hypercyclic if and only if

$(a)$ $|a| > |c|$;

$(b)$ $\D \cap (\co \setminus \Phi(\D)) \ne \emptyset$ and 
$\Dh \cap (\co \setminus \Phi(\D)) \ne \emptyset$.
\end{proposition}

In fact, in \cite{sh} condition $(b)$ is replaced by $\min_{z\in\tz} |\Phi(z)| <1 <
\max_{z\in\tz} |\Phi(z)|$, but this condition is obviously incorrect. If we take
$a=2$, $b=c=0$, then $T_\Phi = 2S^*$ is hypercyclic,
but the estimate $\min_{z\in\tz} |\Phi(z)| <1$ 
does not hold. It is however clear from the proof that the author means the 
correct condition $(b)$. 

Let us show how to deduce this result from our Theorem \ref{main}. It is clear
that $\Phi$ is univalent in $\D$ if and only if $|a|\ge |c|$
and $\Phi$ is univalent in $\overline{\D}$ if and only if $|a| > |c|$. Hence, 
sufficiency of $(a)$ and $(b)$ follows immediately from Statement 2 of Theorem \ref{main}.

Let us show the necessity of $(a)$ and $(b)$. 
Univalence of $\Phi$ implies that $|a|\ge |c|$. To show the strict inequality
we need to apply the argument from \cite{sh}: if $|a| = |c|$, then $T_\Phi$ 
is a normal operator, and hence is not hypercyclic. In general this argument is not applicable.
On the other hand, in \cite{sh} the case $|a|< |c|$ is excluded 
by appealing to the theory of hyponormal operators. It seems that this 
kind of argument can not be used for more general 
operators of the form $T_{\gamma \bar z+\phi(z)}$.

The property $\Dh \cap (\co \setminus \Phi(\D)) \ne \emptyset$ is obvious.
To show that $\D \cap (\co \setminus \Phi(\D)) \ne \emptyset$ one has again
to use an {\it ad hoc} argument from \cite{sh} which                          
uses the very special form of the symbol. Assume, on the contrary, that 
$\sigma(T_\Phi) = \co \setminus \Phi(\D) \subset \{z: |z|\ge 1\}$. Note that in our case
$\sigma(T_\Phi)$ is a convex set (some ellipse) and so it can be separated from the 
unit disc. Thus, there exists $\theta\in \mathbb{R}$ such that 
$\rea(e^{i\theta}\Phi(z)) \ge 1$, $z\in \tz$. Hence, 
$$
\rea (e^{i\theta} T_\Phi f, f)  = \int_{\tz} \rea (e^{i\theta}\Phi) |f|^2 dm 
\ge \|f\|_2^2,
$$ 
and so $T_\Phi$ is an expansion. However, in general, 
$\co \setminus \Phi(\D)$ need not be convex.
\bigskip

\section{Some open questions}

We conclude this note with several open questions.
\medskip
\\
{\bf Question 1.} Let $\Phi = \frac{\gamma}{z} +\phi(z)$ and assume that 
$T_\Phi$ is hypercyclic. Does it follow that 
$$
\D \cap \sigma(T_\Phi) = \D \cap \big( \co \setminus \Phi(\D)\big) \ne \emptyset?
$$
It is true in the case of Toeplitz operators $T_{\overline{\psi}}$ 
with antianalytic symbols 
since if $\sigma(T_{\overline{\psi}}) 
\cap \D =\emptyset$, then $|\psi|>1$ in $\D$ and so 
its inverse $T_{1/\overline{\psi}}$ is a contraction, a contradiction. 
In the case $\Phi(z) = \frac{a}{z} +b +cz$ another argument was suggested by
Shkarin (see Section 4). However, these methods do not seem to apply in general.

Let us mention on the other hand that there are no general obstacles for 
a hypercyclic operator $T$ to satisfy $\sigma(T) \cap\D = \emptyset$ and
the intersection $\sigma(T) \cap\T$ may be a one-point set. Answering a question
of the first author, Sophie Grivaux constructed an example of a hypercyclic operator $T$
such that $\sigma(T)  = \overline{B(2, 1)}$ and $\sigma_p(T) = B(2,1)$ 
(by $B(z_0, r)$ we denote the disc of radius $r$ centered at $z_0$).
\medskip
\\
{\bf Question 2.} Is the univalence of $\Phi$ up to the boundary necessary
in the Statement 2 of Theorem \ref{main}, assuming $\phi\in A(\D)$? Apparently,  
it is necessary for the completeness of the functions of the form $\{h^k\}_{k\ge 0}$
for any individual function $h=\frac{1}{\Phi-\lambda_0}$. However, it seems that 
it is not necessary for completeness of eigenvectors with small or large eigenvalues. 
Namely, the following is true: assume that $\Phi$ is univalent in $\D$ and 
$\co\setminus \overline{\Phi(\D)}$ consists of finite number of connected components 
$U_j$, $1\le j\le m$. If any component $U_j$ intersects $\D$ and $\Dh$, then $T_\Phi$
is hypercyclic. To what extent are these conditions necessary?
\medskip
\\
{\bf Question 3.} What are sufficient conditions of hypercyclicity 
in the case when the valence of $\Phi$ changes inside $\D$? One can show that
the representation  \eqref{bab5} need not be true anymore.
Therefore, it is not clear, which approximation problem corresponds to 
an application of the Godefroy--Shapiro criterion in this case.
\bigskip
\\
{\bf Acknowledgements.}
The authors are grateful to Evgeny Abakumov and  
Sophie Grivaux for useful discussions and to 
the referee for several helpful remarks.

\end{document}